\documentclass[12pt]{amsart}
\usepackage{amssymb,enumerate,mathrsfs, amsmath, amsthm}
\usepackage{url,titletoc,enumitem}

\usepackage[usenames,dvipsnames]{xcolor}
\definecolor{darkblue}{rgb}{0.0, 0.0, 0.55}
\usepackage[pagebackref,colorlinks,linkcolor=BrickRed,citecolor=OliveGreen,urlcolor=darkblue,hypertexnames=true]{hyperref}

\usepackage{cmap}
\usepackage{ericmath}
\usepackage{xspace}
\usepackage{changepage} 

\usepackage[T1]{fontenc}
\usepackage{lmodern}

\usepackage[utf8]{inputenc}
\usepackage{cleveref}
\crefname{section}{§}{§§}

\renewcommand{\subset}{\subseteq}

\newtheorem{theorem}{Theorem}[section]

\newtheorem{lemma}[theorem]{Lemma}

\newtheorem{definition}[theorem]{Definition}

\newtheorem{remark}[theorem]{Remark}

\newtheorem{question}[theorem]{Question}

\newtheorem{lem}[theorem]{Lemma}
\newtheorem{Cor}[theorem]{Corollary}

\newtheorem*{lemma*}{Lemma}

  {
      \theoremstyle{plain}
      
  }

\def\beq{\begin{equation}}
\def\eeq{\end{equation}}

\numberwithin{equation}{section}

\def\beq{\begin{equation}}
\def\eeq{\end{equation}}

\def\hbeta{{\hat{\beta}}}

\def\hgamma{{\hat{\gamma}}}
\def\hY{{\hat{Y}}}
\def\fK{{\mathfrak{K}}}
\def\fE{{\mathfrak{E}}}

\def\arv{\partial^{\mathrm{Arv}}}

\def\comat{\mathrm{co}^\mathrm{mat}}

\def\smlfg{{SM_\ell(\F)^g}}

\newcommand{\egensub}[1]{ {\fE_{#1} (\cD_A)} }
\newcommand{\gensub}[1]{ { \fK_{#1} (\cD_A) } }
\newcommand{\disb}[2]{ { \frak{D}_{#2} (#1) } }

\newcommand{\gengam}[2]{ {\Gamma_{#1,#2} (\cD_A) } }

\def\epsilon{\varepsilon}
\def\eps{\epsilon}
\def\bem{\begin{pmatrix}}
\def\eem{\end{pmatrix}}

\newcommand{\gfs}{generalized free spectrahedron\xspace}

\title[Free Extreme points span generalized free spectrahedra]{Free extreme points span generalized free spectrahedra given by compact coefficients}

\author[E. Evert]{Eric Evert}
\address{Eric Evert, Northwestern University\\
  2233 Tech Drive \\
  Evanston, IL 60208
   }
   \email{eric.evert@northwestern.edu}

\makeindex

\date{\today}

\begin{document}
 
\begin{abstract}
Matrix convexity generalizes convexity to the dimension free setting and has connections to many mathematical and applied pursuits including operator theory, quantum information, noncommutative optimization, and linear control systems. In the setting of classical convex sets, extreme points are central objects which exhibit many important properties. For example, the well-known Minkowski theorem shows that any element of a closed bounded convex set can be expressed as a convex combination of extreme points. Extreme points are also of great interest in the dimension free setting of matrix convex sets; however, here the situation requires more nuance. 

In the dimension free setting, there are many different types of extreme points. Of particular importance are free extreme points, a highly restricted type of extreme point that is closely connected to the dilation theoretic Arveson boundary. If free extreme points span a matrix convex set through matrix convex combinations, then they satisfy a strong notion of minimality in doing so. However, not all closed bounded matrix convex sets even have free extreme points. Thus, a major goal is to determine which matrix convex sets are spanned by their free extreme points. 

Building on a recent work of J. W. Helton and the author which shows that free spectrahedra, i.e., dimension free solution sets to linear matrix inequalities, are spanned by their free extreme points, we show that if one considers linear operator inequalities that have compact operator defining tuples, then the resulting ``generalized'' free spectrahedra are spanned by their free extreme points. In addition, we show that for general matrix convex sets that are closed under complex conjugation, free extreme points spanning over the complexes can be reduced to free extreme points spanning over the reals.

\end{abstract}

\maketitle

\newpage

\section{Introduction}

A linear matrix inequality (LMI) is an inequality of the form
\[
\LA(x) := I- A_1 x_1 - \dots A_g x_g \succeq 0,
\]
where the $A_1,\dots,A_g$ are $d \times d$ symmetric matrices and the $x_1,\dots,x_g$ are real numbers. The solution set of such an LMI, i.e., the set of $x \in \R^g$ such that $\LA(x) \psd 0$, is a convex set called a spectrahedron. Spectrahedra have connections to many areas including convex analysis, optimization, and linear systems control \cite{BEFB94}.

Linear matrix inequalities easily extend to the case where each $X_\ell$ is a $n \times n$ self-adjoint matrix whose product with $A_i$ is the Kronecker product. That is,
\[
\LA(X):= I_{dn} - A_1 \otimes X_1 - \dots - A_g \otimes X_g \psd 0. 
\] 
A {\it free spectrahedron} is the set of all $g$-tuples of self-adjoint matrices (of any size $n \times n$) such that $\LA(X) \psd 0$. The term ``free'' here refers to both the fact that the linear matrix inequality $\LA(X) \psd 0$ is defined independent of the size of the matrices $X_\ell$, and the fact that the corresponding free spectrahedron contains matrix tuples of all sizes $n \times n$. Free spectrahedra naturally arise in problems related to spectrahedral inclusion \cite{HKM13,DDSS17,FNT17,Zal17,HKMS19}, linear control engineering \cite{HMPV09,dOHMP09}, and quantum information \cite{BN18}.

Free spectrahedra are prototypical examples of sets that exhibit a dimension-free type of convexity. Namely, free spectrahedra are {\it matrix convex}, i.e., are closed under {\it matrix convex combinations} where contraction matrices summing to the identity play the role of the convex coefficients.  An important feature of matrix convex combinations is that they allow for combinations of matrix tuples of different sizes. This means that the geometry of any individual level of a matrix convex set is connected to that of all other levels of the set. 

Mirroring the role of extreme points in classical convex sets, extreme points play an important role in the understanding of matrix convex sets \cite{EHKM18,FHL18,PS19,EFHY21,DP22}. However, in the dimension free setting of matrix convexity, there are many types of extreme point. A central goal in the study of matrix convex sets is to determine the most restricted type of extreme point that can recover any element of a given closed bounded matrix convex set through matrix convex combinations. That is, one searches for the strongest possible extension of the classical Minkowski theorem to the dimension free setting. 

Of particular note in this pursuit are matrix extreme points \cite{WW99,F00,F04} and free extreme points \cite{KLS14}. Matrix extreme points are known to span general closed bounded matrix convex sets through matrix convex combinations \cite{K+,HL21}; however, they are not necessarily a minimal spanning set. Free extreme points, on the other hand, are more restricted than matrix extreme points \cite{EHKM18,EEHK22}. Stated informally, a free extreme point of a matrix convex set is an element that can only be expressed via trivial matrix convex combinations of other elements. Free extreme points are of great interest due both to their close connection \cite{EHKM18} to the dilation theoretic Arveson boundary \cite{A69,A08} and to the fact that they necessarily form a minimal spanning set if they do span. The short coming of free extreme points is that, when restricting to finite dimensions, they can fail to span a given closed bounded matrix convex set. In fact, there are closed bounded matrix convex sets which have no (finite dimensional) free extreme points at all \cite{Evert17}. Thus an important question in the pursuit of a dimension free Minkowski theorem is ``which matrix convex sets are the matrix convex hull of their free extreme points''. 

Recently, \cite{EH19} showed that in the case of bounded free spectrahedra, free extreme points span. Furthermore, \cite{EH19} showed that there is a dimension bound on the {\it sum of sizes} of free extreme points needed to express an element of a bounded free spectrahedron as a matrix convex combination of free extreme points. A main contribution of the present article is an extension of this result to ``generalized free spectrahedra''. More precisely, we show that bounded solution sets to linear operator inequalities with compact defining tuples, are the matrix convex hull of their free extreme points. An informal statement of this result is as follows. 

\begin{adjustwidth}{1cm}{} {\it Let $K$ bounded generalized free spectrahedron over $\F=\R$ or $\F=\C$ and assume $K$ is closed under complex conjugation. If $X \in K$ is $g$-tuple of self-adjoint $n \times n$ matrices, then $X$ is a finite matrix convex combination of free extreme points of $K$ whose sum of sizes is bounded above by $n(g+1)$ if $\F=\R$ and by $2n(g+1)$ if $\F=\C$.}
\end{adjustwidth}

\noindent We point the reader to the upcoming Theorem \ref{theorem:AbsSpanGeneralized} for a formal statement of this result. Also see Section \ref{sec:gendefs} for a formal definition of generalized free spectrahedra.

 The proof of Theorem \ref{theorem:AbsSpanGeneralized} follows the same approach used in \cite{EH19} which itself was inspired by works such as \cite{Agl88,DM05,DK15}. Namely, we first work over $\R$ and show the existence of a special type of dilation called a {\it maximal $1$-dilation}. We then show that constructing finite sequence of at most $ng$ maximal $1$-dilations of $X$ results in a so called {\it Arveson boundary point} of $K$. Determining the irreducible components of the resulting Arveson boundary point yields an expression of $X$ as a matrix convex combination of free extreme points of $K$. A fact that enables this extension is that if $A=(A_1,\dots,A_g)$ is a tuple of compact self-adjoint operators and $L_A(X) \succeq 0$, then $L_A(X)$ has a smallest nonzero eigenvalue. 

The extension to the complex setting is achieved by a new result, Theorem \ref{theorem:ComplexReductionToReals}, which shows that for matrix convex sets closed under complex conjugation, real free extreme points span the real part of the set if an only if the set is spanned by its free extreme points over the complexes. Similar to the methodology used in \cite{EH19}, this is accomplished using the standard embedding of tuples of $n \times n$ complex self-adjoint matrices to tuples of $2n \times 2n$ real symmetric matrices together with showing that, for sets closed under complex conjugation, a free extreme point over the reals is also a free extreme point over the complexes. We emphasize that Theorem \ref{theorem:ComplexReductionToReals} is that it is proved for general matrix convex sets closed under complex conjugation, while \cite{EH19} only treated the case of free spectrahedra.

\subsection{Related works}
The study of extreme points of matrix convex sets goes back Arveson who, (in the language of completely positive maps on operator systems) conjectured that if one extends to infinite dimensional levels, then infinite dimensional free extreme points span \cite{A69,A72}. This infinite dimensional question was studied by a number of authors \cite{Ham79,Agl88,MS98,DM05,A08,KLS14,Pas22} until it was finally settled in 2015 by Davidson and Kennedy \cite{DK15}. Work on extreme points in this setting has since continued, e.g., see \cite{DK19} which develops Choquet theory for the infinite dimensional analog of matrix convex sets.

Matrix convexity is also closely related to the rapidly growing area of free analysis \cite{BGM06,Pop06,dOHMP09,Voi10,KVV14,AM15,AJP20,JMS21}. Here the goal is to extend classical geometric and function theoretic results to the noncommutative setting where one considers functions whose inputs are $g$-tuples of matrices or operators. This study was largely pushed by Voiculescu's introduction of free probability which has since been used to great effect in random matrix theory \cite{Voi91,MS17}. Other closely related topics include noncommutative optimization \cite{DLTW08,PNA10,BKP16,MBM21,WM21} and quantum information and games \cite{BN18,BJN22,GLL19,CDN20,PR21}.

\subsection{Reader's guide}
 
 Section \ref{sec:defs} introduces  our definitions and notation and gives a formal statement of our main results, Theorem \ref{theorem:AbsSpanGeneralized} and Theorem \ref{theorem:ComplexReductionToReals}. Section \ref{sec:GenProofs} introduces the notion of a maximal $1$-dilation for generalized free spectrahedra and shows that maximal $1$-dilations in a bounded generalized free spectrahedra reduce the dimension of the dilation subspace. Section \ref{sec:RealVSComplex} shows how to reduce from spanning over the complexes to spanning over the reals for matrix convex sets that are closed under complex conjugation.

\section{Definitions, notation, and main results}
\label{sec:defs}

\def\D{\mathcal{D}}
\newcommand{\M}[3]{M_{#2 \times #3}(\F)^{#1}}

Throughout the article we let $\cH$ denote a Hilbert space over $\F=\R$ or $\F=\C$. $B(\cH)$ and $SA(\cH)$ respectively denote the sets of bounded operators on $\cH$ and bounded self-adjoint operators on $\cH$. Additionally let $SA(\cH)^g$ denote the set of $g$-tuples of the form $X=(X_1,\dots,X_g)$ where each $X_\ell$ is a bounded self-adjoint operator on $\cH$. Say tuples $X, Y \in SA(\cH)^g$ are \df{unitarily  equivalent} if there exists a unitary $U \in B(\cH)$ such that 
\[
U^*XU = (U^* X_1 U,\dots, U^* X_g U) = (Y_1,\dots,Y_g) = Y.
\]
A bounded operator $B \in B(\cH)$ is \df{positive semidefinite} if it is self-adjoint and for any $v \in \cH$, one has $\langle B v , v \rangle \geq 0$. In the case that $B$ is a compact self-adjoint operator, this is equivalent to all of $B$'s eigenvalues being nonnegative. Given two bounded self-adjoint operators $B_1,B_2 \in SA(\cH)$ let $B_1 \psd B_2$ denote that $B_1 - B_2$ is positive semidefinite.

While general Hilbert spaces do play a role in this article, our main interest is matrix convex sets which are restricted to finite dimensions. Thus it is convenient so also have notation for matrix spaces. We let $\M{g}{m}{n}$ denote the set of $g$-tuples of $m\times n$ matrices with entries in $\F$, and let $M_{n}(\F)^g=M_{n\times n}(\F)^g$. Additionally, $\SMF{g}{n}$ is the set of all $g$-tuples of real self-adjoint (symmetric) $n\times n$ matrices. Additionally we set $SM(\F)^g = \cup_n\SMF{g}{n}$. Given a subset $K \subset SM(\F)^g$, we let $K(n)$ denote the set 
\[
K(n) := K \cap SM_{n} (\F)^g.
\]
That is $K(n)$ is set of $g$-tuples of $n \times n$ matrices that are elements of $K$. The set $K(n)$ is called the \df{$n$th level of $K$}.

 For clarity we emphasize that, even when we restrict to real settings, we use the term self-adjoint rather than symmetric to provide consistency in the terminology throughout the article. Similarly, to be consistent with this terminology choice, we use $B^*$ rather than $B^T$ to denote the transpose of a matrix $B$.

\subsection{Linear operator inequalities}

Given a $g$-tuple 
\[
A=(A_1,\dots,A_g) \in SA(\cH)^g
\]
of self-adjoint operators on $\cH$, a \df{monic linear pencil} is a sum of the form
\[
\LA(x) = I_\cH - A_1 x_1 - A_2 x_2 - \dots - A_g x_g.
\]
Given a tuple $X \in \SMF{g}{n}$, the \df{evaluation} of $\LA$ at $X$ is 
\[
\LA(X) = I_{\cH} \otimes I_n - A_1 \otimes X_1 - A_2 \otimes X_2 - \dots - A_g \otimes X_g
\]
where $\otimes$ denotes the Kronecker tensor product. Additionally, let $\lA(X)$ denote the homogeneous linear part of $\LA(X)$. That is,
\[
\lA(X) = A_1 \otimes X_1 + A_2 \otimes X_2 + \dots + A_g \otimes X_g,
\]

The inequality
\[
\LA(X) \psd 0,
\]
is called a \df{linear operator inequality}. If $\cH$ is finite dimension, then the inequality $\LA(X) \psd 0$ is called a \df{linear matrix inequality}. 

\subsection{Free spectrahedra}
\label{sec:gendefs}
Given a $g$-tuple $A\in SA(\cH)^g$ and a positive integer $n$ define the set $\cD_A(n) \subseteq  \SM{g}{n}$ by
\[
\D_A(n) := \{X\in \SMF{g}{n} : \LA(X) \psd 0\}.
\]
That is, $\cD_A(n)$ is the set of all $g$-tuples of $n\times n$ real symmetric matrices $X$ such that the evaluation $\LA(X)$ is positive semidefinite. Additionally define the set $\cD_A \subset \SM{g}{}$ to be the union over all $n$ of the the sets $\cD_A(n)$. That is,
\[
\D_A := \bigcup_{n=1}^{\infty}\D_A(n).
\]

If $\cH$ is finite dimensional, then $\cD_A$ is called a \df{free spectrahedron} and $\cD_A(n)$ is the \df{free spectrahedron $\cD_A$ at level $n$} \cite{EHKM18}. In this article, we extend to considering the case where $\cH$ is a separable infinite dimensional Hilbert space  and $A \in SA(\cH)^g$ is a compact operator tuple. In this case we call $\cD_A$ a \df{\gfs} and we call $\cD_A(n)$ a \df{\gfs at level $n$}. That is, a generalized free spectrahedron $\cD_A$ is the solution set of the linear operator inequality
\[
L_A (X) \succeq 0,
\]
where $A$ is a tuple of self-adjoint compact operators on $\cH$.

\subsection{Matrix Convex Sets}
\newcommand{\set}[1]{\left\{#1\right\}}

Given a finite collection of $g$-tuples $\set{X^i}_{i=1}^\ell$ with $X^i \in \SMF{g}{n_i}$ for each $i = 1,2,\dots,\ell$, a \df{matrix convex combination} of $\set{X^i}_{i=1}^\ell$ is a sum of the form
\[
\sum_{i=1}^{\ell} V_i^* X^i V_i \qquad \text{with} \qquad \sum_{i=1}^{\ell} V_i^* V_i = I_n
\] 
where $V_i \in \M{}{n_i}{n}$ and 
\[
V_i^* X^i V_i = \left(V_i^* X_1^i V_i, V_i^* X_2^i V_i,\dots,V_i^* X_g^i V_i \right) \in \SM{g}{n}
\] 
for all $i = 1,2,\dots,\ell$. A key feature of matrix convex combinations is that the tuples $X^i$ need not be the same size.

A set $K \subseteq SM(\F)^g$ is \df{matrix convex} if it is closed under matrix convex combinations and the \df{matrix convex hull} of $K$, denoted $\comat (K)$, is the set of all matrix convex combinations of the elements of $K$. Equivalently, $K$ is matrix convex if and only if $K = \comat(K)$. It is straightforward to show that generalized free spectrahedra are matrix convex.

Say a matrix convex set $K$ is \df{bounded} if there is some real number $C$ so that
\[
CI_n - \sum_{i=1}^gX_i^2 \psd 0
\]
for all $X = (X_1, X_2,\dots,X_g) \in K(n)$ and all positive integers $n$. It is not difficult to show that $K$ is bounded if and only if $K(1)$ is bounded. Additionally, say $K$ is \df{closed} if $K(n)$ is closed for all $n$. We note that the generalized free spectrahedra considered in this article are, by definition, closed in this sense. Finally, say a matrix convex set $K$ is closed under complex conjugation if $X= (X_1,\dots,X_g) \in K$ implies $\overline{X} = (\overline{X}_1,\dots,\overline{X}_g) \in K$.

\subsection{Free extreme points of matrix convex sets}
As previously mentioned, matrix convex sets have many different types of extreme points. In this article, we restrict our attention to free extreme points. Given a set $K \subseteq \SMF{g}{}$, we say a point $X \in K(n)$ is a \df{free extreme point} of $K$ if whenever $X$ is written as a matrix convex combination 
\[
X = \sum_{i=1}^\ell V_i^* X^i V_i\qquad \text{with} \quad \sum_{i=1}^\ell V_i^* V_i = I_n,
\]
of points $X^i \in K$ with $V_i \neq 0$ for each $i$, then for all $i = 1,2,\dots,\ell$ either $V_i \in M_{n}(\F)$ and $X$ is unitarily equivalent to $X^i$ or $V_i \in \M{}{n_i}{n}$ where $n_i > n$ and there exists $Z^i\in K$ such that $X \oplus Z^i$ is unitarily equivalent to $X^i$. Intuitively, a tuple $X$ is a free extreme point of $K$ if it cannot be written as a nontrivial matrix convex combination of any collection of elements of $K$. We let $\free{K}$ denote the set of all free extreme points of $K$.

\subsection{Free extreme points, dilations, and the Arveson boundary}

We next discuss the connection between free extreme points and the dilation theoretic Arveson boundary. To do this, we must first introduce the notion of an irreducible tuple of matrices. 

Given a matrix $M \in M_{n}(\F)$, a subspace $N \subseteq \F^n$ is a \df{reducing subspace} if both $N$ and $N^\perp$ are invariant subspaces of $M$, A tuple $X\in \SM{g}{n}$ is \df{irreducible} (over $\F$) if the matrices $X_1, \dots, X_g$ have no common reducing subspaces in $\F^n$; a tuple is \df{reducible} (over $\F$) if it is not irreducible. For the remainder of the article, we drop the distinction ``over $\F$'' when referring to irreducible tuples. However, we warn the reader that a real tuple may be reducible over $\C$ even if it is irreducible over $\R$. Thus irreducibility over $\R$ is not equivalent to other well-known definitions of irreducibility which are instead equivalent to irreducibility over $\C$.

\subsubsection{Dilations}

Let $K \subset \smfg$ be a matrix convex set and let $X \in K (n)$. If there exists a positive integer $\ell \in \N$ and $g$-tuples $\beta \in M_{n \times \ell} (\F)^g$ and $\gamma \in \smlfg$ such that
\[
Y=\begin{pmatrix}
X & \beta \\
\beta^* & \gamma
\end{pmatrix}= \left(\begin{pmatrix}
X_1 & \beta_1 \\
\beta^*_1 & \gamma_1
\end{pmatrix}, \cdots,
\begin{pmatrix}
X_g & \beta_g \\
\beta^*_g & \gamma_g 
\end{pmatrix} \right) \in K,
\]
then we say \df{$Y$ is an $\ell$-dilation of $X$}. The tuple $Y$ is a \df{trivial dilation} of $X$ if $\beta=0$. A key connection between matrix convex combinations and dilations is the following. If $Y$ is a dilation of $X$ and $V^*=\begin{pmatrix} I_n & 0 \end{pmatrix},$ then $X=V^* Y V$ with $V^* V=I_n$. That is, if $Y$ is a dilation of $X$, then $X$ can be expressed as a matrix convex combination of $Y$. We note that this matrix convex combination is non-trivial if the dilation itself is non-trivial.

Given a matrix convex set $K$ and an element $X \in K$, we define the \df{dilation subspace of $K$ at $X$}, denoted $\disb{K}{X}$, to be
\[
\disb{K}{X}=\spn\left( \left\{ \beta \in M_{n \times 1}(\F)^g : \mathrm{there \ exists \ a \ } \gamma \in \R^g \mathrm{\ s.t. \ } \begin{pmatrix}
X & \beta \\
\beta^* & \gamma
\end{pmatrix} \in K\right\}\right).
\index{$\disb{K}{X}$}
\]
The dilation subspace obtains its name as, up to scaling, it gives the space of viable dilations of an element of a matrix convex set. 

\begin{lemma}
\label{lem:DilatingBetaAreSubspace}
Let $K \subset \SMF{g}{}$ be a matrix convex set and let $X \in K$. If $\beta \in \disb{K}{X}$, then there exists a real constant $c > 0$ and a tuple $\gamma \in \R^g$ such that
\[
\begin{pmatrix}
X & c\beta \\
c \beta^* & \gamma
\end{pmatrix} \in K.
\]
\end{lemma}
\begin{proof}
If $\beta = 0 \in M_{n \times 1} (\F)^g$, the proof is trivial, so assume $\beta \neq 0$. From the definition of the dilation subspace, the tuple $\beta$ can be expressed as a linear combination $\beta= \sum_{i=1}^k \alpha_i \beta^{i}$ of tuples $\beta^i \in M_{n\times 1}(\F)^g$ where for each $i$, there exists a $\gamma^{i} \in \R^g$ such that
\[
\begin{pmatrix}
X & \beta^i \\
(\beta^i)^* & \gamma^i
\end{pmatrix} \in K.
\]
We first argue that the $\alpha_i$ can be taken to be positive real numbers. To this end, for each $i$ write $\alpha_i =|\alpha_i|e^{i \theta}$ in polar coordinates and observe that 
\[
\begin{pmatrix}
1 & 0 \\
0 & e^{-i\theta}
\end{pmatrix}
\begin{pmatrix}
X & \beta^i \\
(\beta^i)^* & \gamma^i
\end{pmatrix} 
\begin{pmatrix}
1 & 0 \\
0 & e^{i\theta}
\end{pmatrix}
= \begin{pmatrix}
X & e^{i\theta}\beta^i \\
(e^{i\theta}\beta^i)^* & \gamma^i
\end{pmatrix} \in K.
\]
Setting $\tilde{\beta}^i = e^{i\theta}\beta^i$, we obtain $\beta = \sum_{i=1}^k |\alpha_i| \tilde{\beta}^{i}$. Therefore, we can without loss of generality assume $\alpha_i > 0$ as claimed. 

Now, define $c \in \R$ by $1/c = \sum_{i=1}^k \alpha_i$. Then $\sum_{i=1}^k c\alpha_i = 1$, so the convex combination
\[
\sum_{i=1}^k (c\alpha_i) \begin{pmatrix}
X & \beta^i \\
(\beta^i)^* & \gamma^i
\end{pmatrix} = \begin{pmatrix}
X & \sum_{i=1}^k c\alpha_i \beta^i\\
\sum_{i=1}^k c\alpha_i (\beta^i)^* & \sum_{i=1}^k c\alpha_i \gamma^i
\end{pmatrix} =
\begin{pmatrix}
X & c\beta\\
c\beta^* & \sum_{i=1}^k c\alpha_i \gamma^i
\end{pmatrix} .
\]
is an element of $K$, which completes the proof. 
\end{proof}

In the case that $K$ is a free spectrahedron, the definition of the dilation subspace given here is equivalent to the definition given in \cite{EH19}, though the presentation is different, see \cite[Lemma 2.1]{EH19}. The present definition has the advantage that it can be used for general matrix convex sets. See Lemma \ref{lem:DiSubPropertiesCompact} for a more detailed discussion of the dilation subspace of  generalized free spectrahedra. 

\subsubsection{Arveson extreme points}
A tuple $X \in K$ is an \df{Arveson extreme point} of $K$ if $K$ does not contain a nontrivial dilation of $X$. More precisely, $X \in K$ is an Arveson extreme point of $K$ if and only if, if
\beq
\label{eq:ArvDefEq}
\begin{pmatrix}
X & \beta \\
\beta^* & \gamma 
\end{pmatrix} \in K(n+\ell),
\eeq
for some tuples $\beta \in M_{n \times \ell} (\R)^g$ and $\gamma \in \SM{g}{\ell}$, then $\beta=0$. Equivalently, $X$ is an Arveson extreme point of $K$ if and only if $\dim \disb{K}{X} = 0$. If $Y$ is an Arveson extreme point of $K$ and $Y$ is an ($\ell$-)dilation of $X \in K$, then we will say $Y$ is an \df{Arveson ($\ell$-)dilation} of $X$.

The following theorem gives  the connection between free and Arveson extreme points. 
\begin{theorem}[{\cite[Theorem 1.1 (3)]{EHKM18}}]
\label{theorem:EHKMRealComplex} Let $K \subset \SMF{g}{}$ be a matrix convex set. Then $X$ is a free extreme point of $K$ if and only if $X$ is an Arveson extreme point of $K$ and is irreducible. 
\end{theorem}
This theorem has the consequence that the irreducible components of an Arveson extreme point are all free extreme points. This means that expressing an element $X$ of a matrix convex set as a matrix convex combination of free extreme points can be accomplished by finding an Arveson dilation of $X$.  We note that the original proof given in \cite{EHKM18} is done over the complex numbers. The proof over the reals is given in \cite{EH19}.

\subsection{Free extreme points of generalized free spectrahedra}

We are now in position to give a formal statement of our spanning result for generalized free spectrahedra. 

\begin{theorem}
\label{theorem:AbsSpanGeneralized}
Let $\cD_A \subset \SMF{g}{}$ be a bounded generalized free spectrahedron and assume $\cD_A$ is closed under complex conjugation. Given a tuple $X \in \cD_A(n)$, there exists an integer $\ell$ that satisfies
\beq
\label{eq:CarathDimBound}
\ell \leq \dim \disb{\cD_A}{X} \leq ng \qquad \mathrm{if} \ \F=\R,
\eeq
and
\beq
\label{eq:CarathDimBoundComplex}
\ell \leq n+2 \dim\disb{\cD_A}{X} \leq n+2ng \qquad \mathrm{if} \ \F=\C,
\eeq
and a $\ell$-dilation $Y \in \cD_A(n+\ell)$ of $X$ such that $Y$ is an Arveson extreme point of $\cD_A$.

As an consequence, $X$ can be written as a matrix convex combination
\[
X = \sum_{i=1}^k V_i^* F^i V_i  \qquad \mathrm{s.t.} \qquad \sum_{i=1}^k V_i^* V_i = I,
\]
of free extreme points $F^i \in \cD_A(n_i)$ of $\cD_A$ where $\sum_{i=1}^k n_i  \leq n+\ell$. Thus, $\cD_A$ is the matrix convex hull of its free extreme points.
\end{theorem}
\begin{proof}

The proof over $\R$ of the first part of Theorem \ref{theorem:AbsSpanGeneralized} quickly follows from Theorem  \ref{theorem:1DiaReduceKerConDimCompact}. In particular, let $X \in \cD_A$, and set $\ell = \dim \disb{\cD_A}{X}$. Assume that $X$ is not an Arveson extreme point of $\cD_A$, i.e., that $\ell > 0$. Also note that $\ell \leq ng$ since $\disb{\cD_A}{X}$ is a subspace of $M_{n \times 1}(\R)^g$. 

Using Theorem \ref{theorem:1DiaReduceKerConDimCompact} shows that there exists an integer $k \leq \ell$ and a collection of tuples $\{Y^i\}_{i=0}^\ell \subset \cD_A$ such that the following hold:
\begin{enumerate}
\item $Y^0 = X$.
\item For each $i=1,\dots,\ell-1$, the tuple $Y^{i+1}$ is a $1$-dilation of $Y^{i}$ and
\[
\dim \disb{\cD_A}{Y^{i+1}} < \dim \disb{\cD_A}{Y^i}.
\]
\item $Y^\ell$ satisfies  $\dim \disb{\cD_A}{Y^{k}} = 0$.
\end{enumerate}
It follows from Lemma \ref{lem:DiSubPropertiesCompact} \eqref{it:XArvIFFDiSubIsZeroCompact} that $Y$ is an Arveson extreme point of $\cD_A$. Thus $Y^\ell \in \cD_A(n+\ell)$ is an Arveson $\ell$-dilation of $X$. 

The proof over $\R$ of the second part of the theorem quickly follows the first part together with routine dilation theoretic arguments. In particular, one can use the same argument as is used to prove the corresponding statement in \cite[Theorem 1.3]{EH19}.

The proof of the theorem over $\C$ follows by combining the upcoming Theorem \ref{theorem:ComplexReductionToReals} which reduces free extreme points spanning over the complexes to free extreme points spanning over the reals together with Lemma \ref{lem:RealPartOfGenSpec} and Lemma \ref{lem:DilationSubspaceIdentities} \eqref{it:RealEmbedDiSub}. More precisely, Lemma \ref{lem:RealPartOfGenSpec} shows that the real part of a complex generalized free spectrahedron is a indeed a real generalized free spectrahedron while Lemma \ref{lem:DilationSubspaceIdentities} \eqref{it:RealEmbedDiSub} gives the dimension bound of equation \eqref{eq:CarathDimBoundComplex}. \end{proof}

\subsection{Free extreme points of matrix convex sets closed under complex conjugation}

Our second main result shows that, for matrix convex sets that our closed under complex conjugation, free extreme points spanning over the complexes reduces to free extreme points spanning over the reals.

\begin{theorem}
\label{theorem:ComplexReductionToReals}
Let $K \subset SM(\C)^g$ be a matrix convex set such that $K = \overline{K}$, and let $K^\R$ denote $K \cap SM(\R)^g$. Then $K$ is the matrix convex hull of its free extreme points if and only if $K^\R$ is the matrix convex hull of its free extreme points. In notation,
\[
K = \comat(\free(K)) \qquad \mathrm{if \ and \ only \ if} \qquad K^\R = \comat(\free(K^\R)).
\]
Furthermore, given $X = X_\R + i X_\C \in K$ where $X_\R$ and $X_\C$ are the real and imaginary parts of $X$, if the tuple
\[
\begin{pmatrix} X_\R & X_\C \\
X_\C^T & X_\R
\end{pmatrix} \in K^\R,
\]
can be expressed as a matrix convex combination of free extreme points of $K^\R$ whose sum of sizes is at most $m$, then $X$ can also be expressed as a matrix convex combination of free extreme points of $\K$ whose sum of sizes is at most $m$. 
\end{theorem}
\begin{proof}
The main idea in the proof is to show that the Arveson extreme points of $K^\R$ are also Arveson extreme points of $K$. Then, since $X$ is closed under complex conjugation, given $X \in K$, one has that 
\[
\begin{pmatrix} X_\R & X_\C \\
X_\C^T & X_\R
\end{pmatrix} \in K^\R \subset K,
\]
and, moreover, $X$ is a matrix convex combination of this tuple. It follows that if this tuple dilates to an Arveson extreme point $Y \in K^\R$, then $X$ is a matrix convex combination of $Y$ and that $Y$ is an Arveson extreme point of $K$. Details are given in Section \ref{sec:ProofOfReductionToReals}.
\end{proof}

\subsection{Limits of spanning theorems based on maximal 1-dilations}

It is natural to wonder what are the limits of spanning theorems such as Theorem \ref{theorem:AbsSpanGeneralized} may be. We first note that every closed matrix convex set $K \subset \SMF{g}{}$ that has $0$ as an interior point can in fact be expressed as the positivity domain of some operator valued linear pencil $L_A (X)$ where $A \in SA(\cH)^g$ for some appropriate Hilbert space $\cH$. This can be seen by considering the polar dual of $K$. The polar dual of $K$, denoted $K^\circ$, is the set
\[
K^\circ = \{Y \in SM(\F)^g : L_X (Y) \succeq 0 \mathrm{\ for \ all \ } X \in K\}. 
\]
From \cite[Proposition 4.3 (6)]{HKM17}, if $K$ is a closed bounded matrix convex containing $0$, then $K = (K^{\circ})^\circ$, i.e.,
\[
K = \{X \in SM(\F)^g : L_Y (X) \succeq 0 \mathrm{\ for \ all \ } Y \in K^\circ \}.
\]
Moreover, if $0$ is an interior point of $K$, then \cite[Proposition 4.3 (2)]{HKM17} shows that $K^\circ$ is bounded. Setting $A = \oplus_{Y \in K^\circ} Y$ and taking $\cH$ to be the corresponding Hilbert space, one has that $A \in SA(\cH)^g$ and that $K = \{X \in SM(\F)^g: L_A (X) \succeq 0\}$. 

The main role of played by compactness in our proof is to guarantee the existence of spectral gap at zero for certain operators built from the pencil $L_A$, including $L_A(X)$ for tuples $X \in K$. However, our proof easily extends to any matrix convex set $K$ for which the each of the appropriate operators has a spectral gap at zero. This leads to the following question.

\begin{question}
Let $K$ be a closed bounded matrix convex set that contains $0$ as an interior point and set $A = \oplus_{Y \in K^\circ} Y.$ What conditions on $K$ guarantee that $L_A(X)$ has a spectral gap at $0$ for all $X \in K$?
\end{question}

An example of a class of matrix convex sets for which our current approach does not succeed is coordinate projections of free spectrahedra, i.e., free spectrahedrops \cite{HKM16,HKM17}. If this approach were to succeed, then one could always guarantee that an element at level $n$ of a real free spectrahedrop could be expressed as a matrix convex of free extreme points whose sum of sizes is at most $n+ng$. However, there are counter examples that show this is not possible. In particular, let $A \in SM_d (\R)^g$ be an irreducible tuple with $d > g+1$ such that $\cD_A$ is a bounded free spectrahedron. As a consequence of \cite[Theorem 4.11]{HKM17}, if $K$ is the matrix convex hull of $A$, then $K$ is a bounded free spectrahedrop. However, every free extreme point of $K$ is unitary equivalent to $A$, hence lies in $K(d)$. It follows that no element of $K(1)$ can be expressed as a matrix convex combination of free extreme points of $K$ whose sum of sizes is $g+1$ since $g+1<d$. 

Currently it is not known if closed bounded free spectrahedrops are the matrix convex hull of their free extreme points. However, there are several known examples of matrix convex sets that are not the matrix convex hull of their free extreme points. For example, \cite[Example 6.30]{K+} illustrates how one can arrive at a matrix convex set without free extreme points by considering generators of the Cuntz algebra. In fact, this matrix convex set is a free spectrahedron over $\C$ that is not closed under complex conjugation \cite[Remark 2.5]{Pas22}. This illustrates that the assumption that $K$ is closed under complex conjugation in Theorem \ref{theorem:AbsSpanGeneralized} is necessary. Examples of real compact matrix convex sets that have no free extreme points are also known and are discussed in \cite{Evert17}. These examples are essentially obtained by taking the matrix convex hull of an irreducible tuple of compact operators on a real Hilbert space.

\section{Arveson extreme points and generalized free spectrahedra}
\label{sec:GenProofs}

This section introduces maximal $1$-dilations for real generalized free spectrahedra and presents Theorem \ref{theorem:1DiaReduceKerConDimCompact} which shows that maximal $1$-dilations in a real generalized free spectrahedron reduce the dimension of the dilation subspace. Throughout the section $\F=\R$ and $\cH = \ell^2 (\N)$ is a real Hilbert space and $A \in SA(\cH)^g$ is a tuple of compact self-adjoint operators on $\cH$. A simple observation is that with these assumptions, if $X \in \cD_A$ then $L_A (X)$ is diagonalizable and has a smallest nonzero eigenvalue.

\begin{lem}
\label{lem:CompactLOI}
Let $A$ be a $g$-tuple of bounded self-adjoint compact operators on $\cH$ and let $X \in SM_n (\R)^g$ be a $g$-tuple of self-adjoint $n \times n$ matrices. Then 
\[
\Lambda_A (X):=A_1 \otimes X_1 + \dots + A_g \otimes X_g,
\]
is a compact operator self-adjoint on $\cH \otimes \R^n$
\end{lem}
\begin{proof}
Straightforward.
\end{proof}

\begin{lem}
\label{lem:smallestEig}
Let $Q \in B(\cH)$ be a compact self-adjoint operator. Then $I-Q$ is diagonalizable and has a smallest nonzero eigenvalue. 
\end{lem}

\begin{proof}
Since $Q$ is compact and self-adjoint, $Q$ is diagonalizable and only can have zero as a limit point of its spectrum. It follows that 
$I - Q$
is diagonalizable and can only have one as a limit point of its spectrum. Therefore, $I-Q$  has a smallest nonzero eigenvalue. 
\end{proof}

\begin{Cor}
\label{cor:smallestEig}
Let $A$ be a $g$-tuple of compact self-adjoint operators on $\cH$ such that  $\cD_A$ is a generalized free spectrahedron and let $X \in SM_n (\R)^g$ be any $g$-tuple of self-adjoint $n \times n$ matrices. Then $L_A (X) = I-\Lambda_A(X)$ is diagonalizable and has a smallest nonzero eigenvalue. 

As an immediate consequence $L_A (X)^\dagger$ is a bounded self-adjoint operator on $\cH \otimes \R^n$, where $\dagger$ denotes the Moore-Penrose pseudoinverse. 
\end{Cor}

\subsection{The dilation subspace of generalized free spectrahedra}
\label{sec:DiSubpsaceCompact}

We now describe an alternative characterization of the dilation subspace for generalized free spectrahedra. To this end, define the subspace $\gensub{X}$ by
\[
\gensub{X}=\{\beta \in M_{n \times 1} (\R)^g | \ \ker L_A (X) \subset \ker \Lambda_A (\beta^*)\}.
\]
The following lemma shows that $\gensub{X}$ is equal to the dilation subspace of $\cD_A$ at $X$ and explains the connections between $\gensub{X}$ and dilations of $X$. 

\begin{lem}
\label{lem:DiSubPropertiesCompact}
Let $\cD_A \subset \SM{g}{}$ be a generalized free spectrahedron and let $X \in \cD_A (n)$. 
\begin{enumerate}
\item
\label{it:DiBetaIsInDiSubCompact} If $\beta \in M_{n \times 1} (\R)^g$ and 
\[
Y= \begin{pmatrix}
X & \beta \\
\beta^* & \gamma
\end{pmatrix} \in \cD_A (n+1),
\]
is a $1$-dilation of $X$, then $\beta \in \gensub{X}$. 
\item
\label{it:DiSubDilatesXCompact}
Let $\beta \in M_{n \times 1} (\R)^g$. Then $\beta \in \gensub{X}$ if and only if there is a tuple $\gamma \in \cD_A (1)$ real number $c_\gamma>0$ such that
\[
\begin{pmatrix}
X & c_\gamma \beta \\
c_\gamma \beta^* & \gamma
\end{pmatrix} \in \cD_A (n+1).
\]
In particular, one may take $\gamma=0 \in \R^g$.
\item 
\label{it:XArvIFFDiSubIsZeroCompact}
One has
\[
\disb{\cD_A}{X}=\gensub{X}.
\] 
As a consequence, $X$ is an Arveson extreme point of $\cD_A$ if and only if 
\[
\dim \gensub{X}=0.
\]
\end{enumerate}
\end{lem}

\begin{proof}

We first prove Item \eqref{it:DiBetaIsInDiSubCompact}. Multiplying $L_A (Y)$ by permutations, sometimes called canonical shuffles, see \cite[Chapter 8]{P02}, shows 
\beq
\label{eq:2x2CanShufCompact}
L_A (Y) \succeq 0 \quad \mathrm{ if \ and \ only \ if } \quad \begin{pmatrix} L_A (X) & \Lambda_A (\beta) \\ \Lambda_A (\beta^*) & L_A (\gamma) \end{pmatrix} \succeq 0.
\eeq
It follows that if $L_A (Y) \succeq 0$, then for all $u \in \ker L_A(X)$ and $\alpha \in \R$, we have
\begin{align*}
0 &\leq \left\langle \begin{pmatrix} L_A (X) & \Lambda_A (\beta) \\ \Lambda_A (\beta^*) & L_A (\gamma) \end{pmatrix} \begin{pmatrix} u \\
\alpha \Lambda(\beta^*) u \end{pmatrix}, \begin{pmatrix} u \\
\alpha \Lambda(\beta^*) u \end{pmatrix}\right \rangle \\
& = 2\alpha \left\langle \Lambda(\beta^*) u, \Lambda(\beta^*)u \right\rangle + \alpha^2 \langle \lambda_A (\beta) L_A (\gamma) \Lambda (\beta)^* u, u \rangle.
\end{align*}
If $\Lambda(\beta^*)u \neq 0$, then the above would be quadratic polynomial with strictly positive discriminant, hence this expression could not be nonnegative for all $\alpha \in \R$. It follows that that $\Lambda(\beta^*)u$ and that $\ker L_A(X) \subseteq \ker \Lambda_A (\beta^*). $

We now prove Item \eqref{it:DiSubDilatesXCompact}. Note that $L_A (0)=I_\cH$, so using the Schur complement applied to equation \eqref{eq:2x2CanShufCompact} with $\gamma=0$ shows
\[
Y_0=\begin{pmatrix}
X & c\beta \\
c\beta^* & 0
\end{pmatrix} \in \cD_A (n+1)
\]
if and only if
\beq
\label{eq:2x2SchurCompCompactOnZero}
 L_A (X)-c^2 \Lambda_A (\beta) \Lambda_A (\beta^*) \succeq 0.
\eeq
Using Corollary \ref{cor:smallestEig} shows that $L_A (X)$ has a smallest nonzero eigenvalue, so we may pick $c$ small enough so that $\|c^2 \Lambda_A (\beta) \Lambda_A (\beta^*)\|_2$ is less than the smallest nonzero eigenvalue of $L_A (X)$. Furthermore, $\beta \in \gensub{X}$ implies $\ker L_A (X) \subset \ker \Lambda_A (\beta) \Lambda_A (\beta^*)$. Thus this  choice of $c$ guarantees that inequality \eqref{eq:2x2SchurCompCompactOnZero} holds, hence $Y_0 \in \cD_A (n+1)$. The reverse direction is a consequence of Item \eqref{it:DiBetaIsInDiSubCompact}. 

Item \eqref{it:XArvIFFDiSubIsZeroCompact} follows from Items \eqref{it:DiBetaIsInDiSubCompact} and \eqref{it:DiSubDilatesXCompact}.
\end{proof}

\begin{remark}
As in \cite{EH19}, the ability to take $\gamma=0 \in \R^g$ in Lemma \ref{lem:DiSubPropertiesCompact} \eqref{it:DiSubDilatesXCompact} helps simplify the NC LDL$^*$ calculations used in the upcoming proof of Theorem \ref{theorem:1DiaReduceKerConDimCompact}.
\end{remark}

\subsection{Maximal 1-dilations for generalized free spectrahedra.}

Before presenting our next lemma, we introduce some notation. Given a matrix convex set $K \subset \SM{g}{}$ and tuples $X \in \SM{g}{n}$ and $\beta \in M_{n \times 1} (\R)^g$, define
\beq
\label{eq:GammaSet}
\Gamma_{X,\beta}(K):=\left\{ \gamma \in  \R^g : \ \begin{pmatrix}
X & \beta \\
\beta^* & \gamma
\end{pmatrix} \in K \right\}.
\eeq

\begin{lem}
\label{lem:GammaSetCompactConvex}
Let $K \subset \SM{g}{}$ be a closed bounded matrix convex set. Fix $X \in \SM{g}{n}$ and $\beta \in M_{n \times 1} (\R)^g$. Then the set $\Gamma_{X,\beta} (K)$ is a closed bounded convex set. 
\end{lem}
\begin{proof}
If $\Gamma_{X,\beta} (K)$ is empty, then the result trivially holds. When $\Gamma_{X,\beta} (K)$ is nonempty, the fact that $\Gamma_{X,\beta} (K)$ is closed and bounded is immediate from the fact that $K$ is closed and bounded. 

We now show that $\Gamma_{X,\beta} (K)$ is convex. To this end, let $\{\gamma_1,\dots,\gamma_k\} \subset \Gamma_{X,\beta}(K)$ and let $\{c_1,\dots,c_k\}$ be nonnegative constants such that $\sum_{i=1}^k c_i = 1$. For each $i=1\dots,k$ set $V_i = \sqrt{c_i} I_{n+1}. $ Then we have
\[
\begin{pmatrix} X & \beta \\ \beta^* & \sum_{i=1}^k c_i \gamma^i \end{pmatrix} = \sum_{i=1}^k V_i^* \begin{pmatrix} X & \beta \\ \beta^* & \gamma^i \end{pmatrix} V_i \qquad \mathrm{and} \qquad \sum_{i=1}^k V_i^* V_i = 1. 
\]
From the definition of $\Gamma_{X,\beta}(K)$, the above is a matrix convex combination of elements of $K$. Since $K$ is matrix convex, it follows that
\[
\begin{pmatrix} X & \beta \\ \beta^* & \sum_{i=1}^k c_i \gamma^i \end{pmatrix} \in K,
\]
from which we obtain $\sum_{i=1}^k c_i \gamma^i \in \Gamma_{X,\beta} (K)$. That is, $\Gamma_{X,\beta}$ is convex.
\end{proof}

We now present our definition of maximal $1$-dilations for elements of a generalized free spectrahedron.  We mention that our definition is a generalization of the definition of maximal $1$-dilations given in \cite{EH19} which itself was inspired by works such as \cite{DM05}, \cite{A08}, and \cite{DK15}.

\begin{definition}
\label{def:max1diDef} 
Given a bounded generalized free spectrahedron $\cD_A \subset \SM{g}{}$ and a tuple $X \in \cD_A (n)$, say the dilation 
\[
\hat{Y}=\begin{pmatrix}
X & \hbeta \\
\hbeta^* & \hgamma
\end{pmatrix} \in \cD_A (n+1)
\]
is a \df{maximal $1$-dilation} of $X$ if $\hbeta \in M_{n \times 1}(\R)^g$ is nonzero and the following two conditions hold:
\begin{enumerate}
\item \label{it:1isMaximizerCompact} The real number $1$ satisfies
\[
\begin{array}{rllcl}  1 =&\underset{\alpha \in \R, \gamma \in \R^g}{\max} \ \ \ \ \alpha \\
\mathrm{s.t.} & L_A \begin{pmatrix}
X & \alpha \hbeta \\
\alpha \hbeta^* & \gamma
\end{pmatrix} \succeq 0 
\end{array}
\]
\item \label{it:hgamIsExtremeCompact} $\hgamma$ is an extreme point of the closed bounded convex set $\gengam{A}{\hbeta}$ where $\gengam{A}{\hbeta}$ is as defined in equation \eqref{eq:GammaSet}.
\end{enumerate}

\end{definition}

We now show that maximal $1$-dilations in generalized free spectrahedra reduce the dimension of the dilation subspace. 

\begin{theorem}
\label{theorem:1DiaReduceKerConDimCompact}
Let $A \in B(\cH)^g$ be a $g$-tuple of compact self-adjoint operators on $\cH$ such that $\cD_A \subset \SM{g}{}$ is a bounded real generalized free spectrahedron and let $X \in \cD_A (n)$. Assume $X$ is not an Arveson extreme point of $\cD_A$. Then there exists a nontrivial maximal 1-dilation $\hY \in \cD_A (n+1)$ of $X$. Furthermore, any such $\hY$ satisfies 
\[
\dim \disb{\cD_A}{\hY} < \dim \disb{\cD_A}{X}.
\]
\end{theorem}

\begin{proof}

The existence of maximal $1$-dilations in a bounded real generalized free spectrahedron follows from a routine compactness argument together with Lemma \ref{lem:GammaSetCompactConvex}.

Now, let 
\[
\hY=\begin{pmatrix}
X & \hbeta \\
\hbeta^* & \hgamma
\end{pmatrix} ,
\] be a maximal $1$-dilation of $X$. Using Lemma \ref{lem:DiSubPropertiesCompact} \eqref{it:XArvIFFDiSubIsZeroCompact}, it is sufficient to show that
\[
\dim \gensub{\hY} < \dim \gensub{X}.
\]
First consider the subspace
\[
\egensub{\hY}:=\left\{ \eta \in M_{n \times 1} (\R)^g \  : \ \exists \  \sigma\in \R^g \mathrm{\ s.t. \ } \begin{pmatrix} \eta \\
\sigma 
\end{pmatrix} \in \gensub{\hY} \right\}.
\]
\index{$\egensub{\hY}$}
In other words $\egensub{\hY}$ is the projection $\iota$ of $\gensub{\hY}$ defined by
\[
\egensub{Y}:=\iota (\gensub{\hY}) \mathrm{\ where \ } \iota \begin{pmatrix} \eta \\
\sigma 
\end{pmatrix}=\eta,
\]
for $\eta \in M_{n \times 1} (\R)^g$ and $\sigma \in \R^g$. We will show 
\[
\dim \egensub{\hY} < \dim \gensub{X}.
\]

If $\eta \in \egensub{\hY}$, then there exists a tuple $\sigma\in \R^g$ such that 
\[
\begin{pmatrix} \eta^* & \sigma \end{pmatrix} \in \gensub{\hY}.
\]
From Lemma \ref{lem:DiSubPropertiesCompact} \eqref{it:DiSubDilatesXCompact}, it follows that there is a real number $c>0$ so that 
\[
\begin{pmatrix}
X & \hbeta & c\eta \\
\hbeta^* & \hgamma & c\sigma \\
c\eta^* &  c\sigma & 0 
\end{pmatrix}
\in \cD_A.
\]
Since $\cD_A$ is matrix convex it follows that
\[
\begin{pmatrix}
1 & 0 & 0 \\
0 & 0 & 1
\end{pmatrix}
\begin{pmatrix}
X & \hbeta & c\eta \\
\hbeta^* & \hgamma & \sigma \\
c\eta^* &  \sigma & 0 
\end{pmatrix}
\begin{pmatrix}
1 & 0  \\
0 & 0 \\
0 & 1
\end{pmatrix}=
\begin{pmatrix} 
X & c \eta \\
c \eta^* & 0 
\end{pmatrix} \in \cD_A,
\]
so Lemma \ref{lem:DiSubPropertiesCompact} \eqref{it:DiBetaIsInDiSubCompact} shows $\eta \in \gensub{X}$. In particular this shows
\beq
\label{eq:diKerConSubset}
\egensub{\hY} \subset \gensub{X} \qquad \mathrm{hence} \qquad
\dim \egensub{\hY} \leq \dim \gensub{X}.
\eeq

Now, assume towards a contradiction that 
\[
\dim \gensub{\hY} \geq \dim \gensub{X}.
\]
We next show that this implies that there is a real number $c$ and a tuple $\sigma \in \R^g$ such that
\beq
\label{eq:badBetaDi}
L_A \begin{pmatrix}
X & \hat{\beta} & c \hat{\beta} \\
\hat{\beta}^* & \hat{\gamma} & \sigma \\
c\hat{\beta}^* & \sigma & 0
\end{pmatrix}\succeq 0,
\eeq
and such that either $ c \neq 0$ or $\sigma \neq 0$. To see this, observe that equation \eqref{eq:diKerConSubset} implies that if $\dim \egensub{\hY} = \dim \gensub{X}$, then we must have 
\[
\egensub{\hY} = \gensub{X}.
\]
In this case, using Lemma \ref{lem:DiSubPropertiesCompact} \eqref{it:DiSubDilatesXCompact} shows that there is a nonzero $c \in \R$ and some (possibly zero) $\sigma \in \R^g$ such that inequality \eqref{eq:badBetaDi} holds. On the other hand, if $\dim \egensub{\hY} < \dim \gensub{X}$, then there must exist tuples $\eta \in M_{n \times 1} (\R)^g$ and $\sigma^1, \sigma^2 \in \R^g$ such that $\sigma^1 \neq \sigma^2$ and so
\beq
\label{eq:hsigmaKerCon}
\begin{pmatrix}
\eta \\
\sigma^1
\end{pmatrix},
\begin{pmatrix}
\eta \\
\sigma^2
\end{pmatrix} \in \gensub{Y}, \qquad \mathrm{hence} \qquad \begin{pmatrix}
0 \\
\sigma^1-\sigma^2
\end{pmatrix} \in \gensub{Y}.
\eeq
In this case, setting $\sigma = \alpha(\sigma^1 - \sigma^2)$ for an appropriately chosen constant $\alpha > 0$ and again using using Lemma \ref{lem:DiSubPropertiesCompact} \eqref{it:DiSubDilatesXCompact} shows that there is a (possibly zero) real number $c$ and a nonzero tuple $\sigma$ such that inequality \eqref{eq:badBetaDi} holds.

We now use inequality \eqref{eq:badBetaDi} together with the NC LDL$^*$-decomposition to show that $\hY$ cannot be a maximal $1$-dilation of $X$, a contradiction to our definition of $\hY$. Applying the NC LDL$^*$-decomposition (up to canonical shuffles) shows that inequality \eqref{eq:badBetaDi} holds if and only if $L_A (X) \succeq 0$ and the Schur complements
\beq
\label{eq:3LDLCompactMiddleCompact}
I_\cH-c^2 Q \succeq 0
\eeq
and
\beq
\label{eq:3LDLCompact}
L_A (\hgamma)-Q-\left(\Lambda_A(\sigma)-cQ\right)^*\left(I_\cH-c^2 Q\right)^\dagger \left(\Lambda_A(\sigma)-cQ\right) \succeq 0
\eeq
where 
\beq
\label{eq:QdefCompact}
Q:= \Lambda_A (\hbeta^*) L_A (X)^\dagger \Lambda_A (\hbeta).
\eeq
It follows that
\beq
\label{eq:gammaSchurCompact}
L_A (\hat{\gamma})-Q \succeq 0
\eeq
and
\beq
\label{eq:3diaKerContainCompact}
\ker  [L_A (\hat{\gamma})-Q] \subseteq \ker [\Lambda_A(\sigma)-c Q].
\eeq

Recall from Corollary \ref{cor:smallestEig} and Lemma \ref{lem:CompactLOI} that $L_A (X)^\dagger$ is a bounded self-adjoint operator and $\Lambda_A (\hbeta)$ and $\Lambda_A (\hbeta^*)$ are compact operators. It follows that $Q$ is a compact self-adjoint operator. Therefore 
\[
L_A (\hat{\gamma})-Q=I_\cH-(\Lambda_A (\hat{\gamma})+Q)
\]
is the identity minus a compact self-adjoint operator and by Lemma \ref{lem:smallestEig} has a smallest nonzero eigenvalue. Therefore, picking $\tilde{\alpha} > 0$ so that $\tilde{\alpha}\|\Lambda_A(\sigma)-c Q\|$ is smaller than the smallest nonzero eigenvalue of $L_A (\hat{\gamma})-Q$ and using inequalities \eqref{eq:gammaSchurCompact} and \eqref{eq:3diaKerContainCompact} guarantees
\[
L_A (\hat{\gamma})-Q \  \pm \ \tilde{\alpha} \left(\Lambda_A(\sigma)-c Q\right) \succeq 0.
\]
It follows from the above that
\beq
\label{eq:betterBetaSchurCompact}
\begin{array}{rllcl}
&L_A (\hat{\gamma} \pm \tilde{\alpha} \sigma)- (1 \pm c\tilde{\alpha} )Q\\
=& L_A (\hat{\gamma} \pm \alpha \sigma)-\left(\Lambda_A (\sqrt{1 \pm c\tilde{\alpha} }\hat{\beta}^*)L_A(X)^\dagger\Lambda_A(\sqrt{1 \pm c\tilde{\alpha} }\hat{\beta})\right)& \succeq & 0.
\end{array}
\eeq
Since $L_A(X) \succeq 0$, equation \eqref{eq:betterBetaSchurCompact} implies
\beq
\label{eq:betterBetaDiCompact}
L_A \begin{pmatrix}
X & \sqrt{1 \pm c\tilde{\alpha} }\hat{\beta} \\
\sqrt{1 \pm c\tilde{\alpha} }\hat{\beta}^* & \hat{\gamma}\pm \tilde{\alpha}  \sigma
\end{pmatrix}\succeq 0.
\eeq
Recalling that $\hY$ is a maximal $1$-dilation of $X$, we must have 
\[
\sqrt{1 \pm c\tilde{\alpha} }  \leq 1.
\]
It follows that $c\tilde{\alpha} =0$. Since $\tilde{\alpha} > 0$, it follows that $c = 0$. From our construction, this in turn implies that $\sigma \neq 0$. But then equation \eqref{eq:betterBetaDiCompact} implies that 
\[
\hat{\gamma}\pm \tilde{\alpha}\sigma \in \gengam{X}{\hbeta},
\] 
which contradicts the fact that $\hat{\gamma}$ is an extreme point of the convex set $\gengam{X}{\hbeta}$. We conclude that
\[
\dim \gensub{\hY} < \dim \gensub{X},
\]
from which we can use Lemma \ref{lem:DiSubPropertiesCompact} \eqref{it:XArvIFFDiSubIsZeroCompact} to conclude that
\[
\disb{\cD_A}{\hY} < \disb{\cD_A}{X},
\]
as claimed.
\end{proof}

\section{From reals to complexes}
\label{sec:RealVSComplex}

We now show how to reduce spanning over the complexes to spanning over the reals. The main goal of the section is to generalize \cite[Section 3]{EH19} which considers free spectrahedra to general matrix convex sets. We note that while \cite[Section 3]{EH19} restricted to free spectrahedra, many of the proofs there only relied on matrix convexity and immediately generalize to our present setting. Proofs of such results will be given in an online appendix.

We begin by relating the Arveson extreme points of a matrix convex set over the complexes to the Arveson extreme points of the associated matrix convex set over the reals. 

\begin{lemma}
\label{lemma:RealVsComplexArvBoundary}
Let $K \subset SM(\C)^g$ be a matrix convex set such that $K = \overline{K}$, and let $K^\R$ denote $K \cap SM(\R)^g$. 
\begin{enumerate}
	\item \label{it:ArvConjClosed} $X \in \arv(K)$ if and only if $\overline{X} \in \arv(K)$.
	\item \label{it:RealArvIsComplexArv} If $X \in K^\R$, then $X \in \arv(K)$ if and only if $X \in \arv(K^\R)$.
\end{enumerate}  
\end{lemma}

\begin{proof}
See Appendix \ref{App:RealComplexArvBoundaryProof}.
\end{proof}

\begin{remark}
It may not be the case that the real free extreme points of $K^\R$ are free extreme points of $K$. This is due to the issue of reducibility over $\R$ vs $\C$. For example, if $X \in K$ is a free extreme point of $K$ with complex part $X_\C \neq 0$, then the tuple $U_m (X \oplus \overline{X})U$ is reducible over $\C$ while it is irreducible over $\R$. Therefore, $U_m (X \oplus \overline{X})U$ is a free extreme point of $K^\R$, but not a free extreme point of $K$. 
\end{remark}

\subsection{Proof of Theorem \ref{theorem:ComplexReductionToReals}}

\label{sec:ProofOfReductionToReals}

The main tool we use in this proof is the standard embedding of a tuple of complex self-adjoint matrices to a tuple of real symmetric matrices. To this end, we define $U_m$ to be the unitary
\begin{equation}
\label{eq:EmbeddingUnitary}
U_m:= \frac{\sqrt{2}}{2}
\begin{pmatrix}
I_{m} & -i I_{m} \\
-i I_{m} & I_{m} 
\end{pmatrix} \in M_{2m} (\C).
\end{equation}
A key fact is that if $X \in \SMC{g}{m}$ with real and imaginary parts $X_\R$ and $X_\C$ then
\[
U_m^* (X \oplus \overline{X}) U_m = \begin{pmatrix}
X_\R & X_\C \\
X_\C^T & X_\R
\end{pmatrix}.
\]
Note that the above expansion makes use of the fact that $X_\C = -X_\C^T$ is a tuple of skew symmetric matrices.
\begin{proof}

We first show that that $K = \comat(\free(K))$ implies $K^\R  = \comat(\free(K^\R ))$. To this end, let $X \in \K^R \subset \K$. By assumption, there exists a tuple
\[
Y = \begin{pmatrix}
X & \beta \\
\beta^* & \gamma
\end{pmatrix} \in K(m),
\]
for some appropriate $m$ such that $Y \in \arv(K)$. Since $K$ is closed under complex conjugation $\overline{Y} \in K$. Moreover, we have that $X$ dilates to
\[
U_{m}^*
(Y \oplus \overline{Y})
U_{m} = \begin{pmatrix} Y_\R & Y_\C \\
Y_\C^T & Y_\R
\end{pmatrix}
=\begin{pmatrix}
X & \beta_\R & 0 & \beta_\C \\
\beta_\R^T & \gamma_\R & -\beta_\C^T & \gamma_\C  \\
0 & -\beta_\C & X & \beta_\R \\
\beta_\C^T & \gamma_\C^T & \beta_\R^T & \gamma_\R
\end{pmatrix} \in K^\R(2m),
\]
As the Arveson boundary is closed under unitary conjugation and since the above tuple is real valued, we can apply Lemma \ref{lemma:RealVsComplexArvBoundary} \eqref{it:RealArvIsComplexArv} to conclude that
\[
U_{m}^*
(Y \oplus \overline{Y})
U_{m} \in \arv(K^\R),
\]
is an Arveson dilation of $X$ in $K^\R$. It follows that $K^\R= \comat(\free(K^\R))$.

We now show that $K^\R  = \comat(\free(K^\R ))$ implies $K = \comat(\free(K))$. To this end let $X \in K(n)$. Since $K$ is closed under complex conjugation, we have $X \oplus \overline{X}$ in $K$, hence
\[
Z:= U_n^*
\begin{pmatrix}
X & 0 \\
0 & \overline{X}
\end{pmatrix}
U_n = \begin{pmatrix} X_\R & X_\C \\
X_\C^T & X_\R
\end{pmatrix}
\in{K^\R}. 
\]
By assumption $\K^R$ is the matrix convex hull of its free extreme points, so $Z$ dilates to an Arveson extreme point $Y \in \arv(K^\R (m))$ for some appropriate $m$. Again using Lemma \ref{lemma:RealVsComplexArvBoundary} \eqref{it:RealArvIsComplexArv} shows that $Y \in \arv(K)$. By decomposing $Y$ into its irreducible (over $\C$) components, we obtain an expression for $X$ as matrix convex combination of free extreme points of $K$ whose sum of sizes is at most $m$.
\end{proof}

\subsection{The real part of a complex generalized free spectrahedron}

To use Theorem \ref{theorem:ComplexReductionToReals} to reduce to the real case in the proof of Theorem \ref{theorem:AbsSpanGeneralized}, we need to prove that the real part of a complex generalized free spectrahedron is indeed a real generalized free spectrahedron. This is handled by the following lemma.

\begin{lemma}
\label{lem:RealPartOfGenSpec}
Let $\cD_A \subset \SMC{g}{}$ be a complex generalized free spectrahedron. Then $\cD_A^\R = \cD_A \cap \SM{g}{}$ is a real generalized free spectrahedron. 
\end{lemma}
\begin{proof}
The proof follows the same approach as the proof of \cite[Lemma 3.3]{EH19}. Details are given in Appendix \ref{App:RealPartOfGenSpecProof}. \end{proof}

Finally, for the dimension bound on the sum of sizes of free extreme points required for a free Caratheodory expansion that is presented in Theorem \ref{theorem:AbsSpanGeneralized}, we extend \cite[Lemma 3.2]{EH19} from free spectrahedra to general matrix convex sets.

\begin{lemma}
\label{lem:DilationSubspaceIdentities}
Let $K \subset \SMF{g}{}$ be a matrix convex set. Then we have the following:
\begin{enumerate}
\item \label{it:DirectSumDiSub} Let $X \in K(n)$ and $Y \in K(m)$. Then
\[
\disb{K}{X \oplus Y}  = \left\{\left( \beta^* \ \sigma^* \right)^* \in M_{(n+m) \times 1} (\F)^g | \ \beta \in \disb{K}{X} \mathrm{\ and \ } \sigma \in \disb{K}{Y} \right\},
\]
and
\[
\dim \disb{K}{X \oplus Y} = \dim \disb{K}{X}+\dim \disb{K}{Y}.
\]

\item \label{it:UnitaryDiSub} Let $X \in K(n)$ and let $U \in M_n (\F)$ be a unitary. 
\[
\disb{K}{X} = U^* \disb{K}{U^*XU} \qquad \mathrm{and} \qquad \dim \disb{K}{X} = \dim \disb{K}{U^* X U}.
\]

\item \label{it:ConjugateDiSub} If $K$ is closed under complex conjugation, then for any $X \in K$, 
\[
\disb{K}{X} = \overline{\disb{K}{\overline{X}}} \qquad \mathrm{and} \qquad \dim \disb{K}{X} = \dim \disb{K}{\overline{X}}.
\]

\item \label{it:RealEmbedDiSub} If $K$ is closed under complex conjugation, then for any $X \in K(n)$, \[\dim \disb{K^\R}{U_n^* \left(X \oplus \overline{X}\right)U_n} \leq 2\dim \disb{K}{X}, \]
where $U_n$ is the unitary defined in equation \eqref{eq:EmbeddingUnitary}.
\end{enumerate}
\end{lemma}
\begin{proof}
See Appendix \ref{App:DiSubIdProofs}.
\end{proof}

\bibliographystyle{siamplain}
\bibliography{../References/MatConvexRefs}

\newpage

\section{Appendix}
\label{sec:RealVsComplexProofs}

The appendix contains proofs of several of the lemmas presented in Section \ref{sec:RealVSComplex}

\subsection{Proof of Lemma \ref{lemma:RealVsComplexArvBoundary}}
\label{App:RealComplexArvBoundaryProof}

\begin{proof}
The proof follows the same approach as \cite[Lemma 3.1]{EH19}. We first prove item \eqref{it:ArvConjClosed}. Since $K$ is closed under complex conjugation, for any $\beta \in M_{n \times 1} (\C)^g$ and $\gamma \in \R^g$, we find
\[
 \begin{pmatrix}
X & \beta \\
\beta^* & \gamma
\end{pmatrix} \in K \qquad \mathrm{if \ and \ only \ if} \qquad Y = \begin{pmatrix}
\overline{X} & \overline{\beta} \\
\overline{\beta}^* & \overline{\gamma}
\end{pmatrix} \in K.
\]
Therefore, there is a nontrivial dilation of $X$ in $K$ if and only if there is a nontrivial dilation of $\overline{X}$ in $K$. 

We next prove item \eqref{it:RealArvIsComplexArv}. Let $X \in K^\R$ and observe if $X$ cannot be trivially dilated over $\C$, then $X$ also cannot be trivially dilated over $\R$, hence it is clear that $X \in \arv (K)$ implies $X \in \arv(K^\R)$. Now assume $X \in \arv(K^\R)$ and suppose $X$ dilates (over $\C$) to
\[
Y = \begin{pmatrix} X & \beta \\
\beta^* & \gamma \end{pmatrix} \in K(m).
\]
We will show that $\beta=0$. 

Since $K$ is closed under complex conjugation, $\overline{Y} \in K(m)$,
hence $X$ dilates to $Y \oplus \overline{Y} \in K(2m)$. Now, let $Y_\R,\beta_\R,\gamma_R$ and $Y_\C,\beta_\C,\gamma_C$ denote the real and imaginary parts of $Y,\beta,\gamma$, respectively. Then since $K$ is closed under unitary conjugation, we obtain
\[
U_m^*
\begin{pmatrix}
Y & 0 \\
0 & \overline{Y}
\end{pmatrix}
U_m = \begin{pmatrix} Y_\R & Y_\C \\
Y_\C^T & Y_\R
\end{pmatrix}
=\begin{pmatrix}
X & \beta_\R & 0 & \beta_\C \\
\beta_\R^T & \gamma_\R & -\beta_\C^T & \gamma_\C  \\
0 & -\beta_\C & X & \beta_\R \\
\beta_\C^T & \gamma_\C^T & \beta_\R^T & \gamma_\R
\end{pmatrix} \in K^\R(2m).
\]
Here we have used the fact that $X$ is real valued and that $Y_\C$ is skew-symmetric, which follows from the fact that $Y$ is self-adjoint. Now, since $X$ is an Arveson extreme point of $K^\R$, we know that $X$ cannot be nontrivially dilated in $K^\R$, hence $\beta_\R =\beta_\C =0$. From this we obtain $\beta=0$ which proves that $X$ is an Arveson extreme point of $K$. 
\end{proof}

\subsection{Proof of Lemma \ref{lem:RealPartOfGenSpec}}
\label{App:RealPartOfGenSpecProof}

\begin{proof}
Let $\cD_A \subset \SMC{g}{}$ be a generalized free spectrahedron and without loss of generality take $\cH = \ell^2(\mathbb{N})$ as a Hilbert space over $\C$. Then, since $\cD_A$ is closed under complex conjugation, for $X \in M_n (\R)$, we have
\[
L_A (X) \succeq 0\qquad  \mathrm{if\ and \ only \ if} \qquad L_{\overline{A}} (X) \succeq 0,
\]
where $\overline{A} \in B(\ell^2(\mathbb{N}))$ is defined by  $\overline{A}v = \overline{A \overline{v}}$ for all $v \in \ell^2(\mathbb{N})$. Letting $U_\cH$ denote the unitary
\[
\frac{\sqrt{2}}{2} \begin{pmatrix}
I_\cH &-i I_\cH \\
-I_\cH & I_\cH
\end{pmatrix} \in B(\ell^2(\N)\oplus \ell^2(\N)),
\]
we find that each element of the compact operator tuple $U_\cH^*(A \oplus \overline{A}) U_\cH$ maps the real part of $\ell^2(\N)\oplus \ell^2(\N)$ to the real part of $\ell^2(\N)\oplus \ell^2(\N)$. Therefore, we can treat $U_\cH^*(A \oplus \overline{A}) U_\cH$ as a tuple of bounded compact operators on $\ell^2(\N)\oplus \ell^2(\N)$ viewed a Hilbert space over $\R$. From this we obtain that $\cD_A^\R = \cD_{U_\cH^*(A \oplus \overline{A}) U_\cH}^\R$ is a real generalized free spectrahedron.
\end{proof}

\subsection{Proof of Lemma \ref{lem:DilationSubspaceIdentities}}
\label{App:DiSubIdProofs}

\begin{proof}

We first prove Item \eqref{it:DirectSumDiSub}. When working with a generalized free spectrahedron, one can apply Lemma \ref{lem:DiSubPropertiesCompact} together with the strategy used to prove \cite[Lemma 3.2 (1)]{EH19}. However, an alternative approach is needed for general matrix convex sets. For this case, Let $X \in K(n)$ and $Y \in K(m)$ and let $\beta \in \disb{K}{X}$ and $\eta \in \disb{K}{Y}$. Then using Lemma \ref{lem:DilatingBetaAreSubspace}, there exist constants $c_1,c_2 > 0$ and tuples $\gamma,\sigma \in \R^g$ such that
\[
\begin{pmatrix}
X & 0 & c_1\beta \\
0 & Y & 0 \\
c_1 \beta^* & 0 & \gamma
\end{pmatrix} \in K \qquad \mathrm{and} \qquad
\begin{pmatrix}
X & 0 & 0 \\
0 & Y & c_2\eta \\
0 & c_2\eta^* & \sigma
\end{pmatrix} \in K.
\]
It follows that
\[
\begin{pmatrix}
\beta^* &
0
\end{pmatrix},\begin{pmatrix}
0 &
\eta^*
\end{pmatrix} \in \disb{X \oplus Y}{K}, 
\]
from which we obtain
\[
\disb{K}{X \oplus Y}  \supseteq \left\{\left( \beta^* \ \eta^* \right)^* \in M_{(n+m) \times 1} (\F)^g | \ \beta \in \disb{K}{X} \mathrm{\ and \ } \eta \in \disb{K}{Y} \right\}.
\]

For the reverse containment, suppose $\left( \beta^* \ \eta^* \right)^* \in \disb{X \oplus Y}{K}$. Again using Lemma \ref{lem:DilatingBetaAreSubspace} shows there is a constant $c > 0$ and a tuple $\gamma \in \R^g$ such that
\[
\begin{pmatrix}
X & 0 & c \beta \\
0 & Y & c\eta \\
c \beta^* & c\eta^* & \sigma
\end{pmatrix} \in K.
\]
Using the matrix convexity of $K$, we obtain
\[
\begin{pmatrix}
1 & 0 & 0 \\
0 & 0 & 1
\end{pmatrix}\begin{pmatrix}
X & 0 & c \beta \\
0 & Y & c\eta \\
c \beta^* & c\eta^* & \gamma
\end{pmatrix} 
\begin{pmatrix}
1 & 0 \\
0 & 0 \\
0 & 1
\end{pmatrix} = \begin{pmatrix}
X & c \beta \\
c \beta^* & \gamma
\end{pmatrix}  \in K,
\]
and similarly
\[
\begin{pmatrix}
0 & 1 & 0 \\
0 & 0 & 1
\end{pmatrix}
\begin{pmatrix}
X & 0 & c \beta \\
0 & Y & c\eta \\
c \beta^* & c\eta^* & \gamma
\end{pmatrix} 
\begin{pmatrix}
0 & 0 \\
1 & 0 \\
0 & 1
\end{pmatrix} = \begin{pmatrix}
X & c \eta \\
c \eta^* & \gamma
\end{pmatrix}  \in K.
\]
We conclude $\beta \in \disb{K}{X}$ and $\eta \in \disb{K}{Y}$, which proves the reverse containment.

The proof of Items \eqref{it:UnitaryDiSub} and \eqref{it:ConjugateDiSub} use the same approach of \cite[Lemma 3.2 (2) and (3)]{EH19} together with an application of Lemma \ref{lem:DilatingBetaAreSubspace}. In particular, for a unitary $U \in M_n(\F)$, we have
\[
\begin{pmatrix} 
X & \beta \\
\beta^* & \gamma
\end{pmatrix} \in K \iff
\begin{pmatrix}
U^* X U & U^* \beta \\
\beta^* U & \gamma
\end{pmatrix} =
\begin{pmatrix}
U^* & 0 \\
0 & 1
\end{pmatrix}
\begin{pmatrix} 
X & \beta \\
\beta^* & \gamma
\end{pmatrix}
\begin{pmatrix}
U & 0 \\
0 & 1
\end{pmatrix} \in K
\]
Similarly, if $K$ is closed under complex conjugation, then we immediately have
\[
\begin{pmatrix} 
X & \beta \\
\beta^* & \gamma
\end{pmatrix} \in K \iff \begin{pmatrix} 
\overline{X} & \overline{\beta} \\
\overline{\beta^*} & \overline{\gamma}
\end{pmatrix} \in K.
\]
Using the above with Lemma \ref{lem:DilatingBetaAreSubspace} completes the proof of Items \eqref{it:UnitaryDiSub} and \eqref{it:ConjugateDiSub}.

Item \eqref{it:RealEmbedDiSub} is an immediate consequence of Items \eqref{it:DirectSumDiSub}, \eqref{it:UnitaryDiSub}, and \eqref{it:ConjugateDiSub}. \end{proof}

\end{document}